\documentclass[12pt]{scrartcl}

\usepackage[]{amsmath, amssymb,amsfonts,amsthm,mathtools,geometry}
\usepackage{braket}
\usepackage{enumerate}
\usepackage{color}
\usepackage{ytableau}
\makeatletter
\DeclareOldFontCommand{\rm}{\normalfont\rmfamily}{\mathrm}
\DeclareOldFontCommand{\sf}{\normalfont\sffamily}{\mathsf}
\DeclareOldFontCommand{\tt}{\normalfont\ttfamily}{\mathtt}
\DeclareOldFontCommand{\bf}{\normalfont\bfseries}{\mathbf}
\DeclareOldFontCommand{\it}{\normalfont\itshape}{\mathit}
\DeclareOldFontCommand{\sl}{\normalfont\slshape}{\@nomath\sl}
\DeclareOldFontCommand{\sc}{\normalfont\scshape}{\@nomath\sc}
\makeatother

\usepackage{tikz}
\tikzset{v/.style={
  circle, draw, inner sep=2pt, minimum size=6pt, fill=white}}

\usepackage{shuffle}

\theoremstyle{plain}% default
\newtheorem{theorem}{Theorem}[section]
\newtheorem{lemma}[theorem]{Lemma}
\newtheorem{proposition}[theorem]{Proposition}
\newtheorem{corollary}[theorem]{Corollary}

\theoremstyle{definition}
\newtheorem{definition}[theorem]{Definition}
\newtheorem{conjecture}[theorem]{Conjecture}
\newtheorem{example}[theorem]{Example}

\newtheorem{problem}[theorem]{Problem}

\theoremstyle{remark}
\newtheorem*{remark}{Remark}

\DeclareMathOperator{\Hom}{Hom}
\DeclareMathOperator{\QSym}{QSym}
\DeclareMathOperator{\wt}{wt}
\DeclareMathOperator{\LT}{LT}
\DeclareMathOperator{\type}{type}
\DeclareMathOperator{\St}{St}
\DeclareMathOperator{\jump}{jump}
\def\N{{\sf N}}

\title {Order Quasisymmetric Functions Distinguish Rooted Trees}
\author{
Takahiro Hasebe\thanks{Department of Mathematics, Hokkaido University, North 10, West 8, Kita-ku, Sapporo 060-0810, JAPAN E-mail: thasebe@math.sci.hokudai.ac.jp} and 
Shuhei Tsujie\thanks{Department of Mathematics, Hokkaido University, North 10, West 8, Kita-ku, Sapporo 060-0810, JAPAN E-mail: tsujie@math.sci.hokudai.ac.jp}
}
\date{}
\begin{document}
\maketitle
	
\begin{abstract}
Richard P. Stanley conjectured that finite trees can be distinguished by their chromatic symmetric functions.  
In this paper, we prove an analogous statement for posets: Finite rooted trees can be distinguished by their order quasisymmetric functions.  
\end{abstract}


{\footnotesize {\it 2010 MSC:}
06A11; %Algebraic aspects of posets
06A07, %Combinatorics of partially ordered set
05A05, %Permutations, words, matrices
05C15, %Coloring of graphs and hypergraphs

% 05Cxx Graph theory
% 05Axx Enumerative combinatorics
% 06Axx Ordered sets

}

\section{Introduction}
Finding graph isomorphisms is known to be a hard problem. A typical mathematical way of understanding isomorphisms between objects is to compute invariants, not only for graphs but for other objects like manifolds and algebras. A well known invariant of graphs is the \textbf{chromatic polynomial} which counts the number of homomorphisms from a graph into the complete graph $K_n$ on $n$ vertices.  
More precisely, let $ G=(V_{G},E_{G}) $ and $ H=(V_{H},E_{H}) $ be simple graphs. 
A \textbf{homomorphism} from $ G $ to $ H $ is a map $ f \colon V_{G} \rightarrow V_{H} $ such that $ \{u,v\} \in E_{G} $ implies $ \{f(u),f(v)\} \in E_{H} $. 
Let $ \Hom(G,H) $ denote the set of homomorphisms from $ G $ to $ H $. A homomorphism $f \in \Hom(G,K_n)$, also called a proper coloring of $G$, is a map $f\colon V_G \to \{1,\dots, n\}$ such that any neighboring two vertices of $G$ are mapped to different numbers (colors). The cardinality of such maps $\#\Hom(G,K_n)$ turns out to be a polynomial on $n$ of degree $\#V_G$, meaning that there exists a (unique) polynomial $\chi(G,t)$ of degree $\#V_G$ such that $\chi(G,n)=\#\Hom(G,K_n)$ holds for all $n \in \mathbb{N}$. This is called the chromatic polynomial. 

Given an invariant, a basic question is to what extent the invariant distinguishes the objects of interest. 
For example, from the chromatic polynomial one can extract the number of vertices, the number of edges and the number of connected components. On the other hand, all trees on $m$ vertices have the same chromatic polynomial $t(t-1)^{m-1}$, and so, the chromatic polynomial cannot distinguish trees at all.  The converse is also known: if a finite simple graph has the chromatic polynomial $t(t-1)^{m-1}$ then the graph is a tree on $m$ vertices. These results can be found in the introductory article \cite{Read68}. 

Richard P.\ Stanley \cite{stanley1995symmetric} introduced an invariant called the \textbf{chromatic symmetric function}, which is stronger than the chromatic polynomial. For a finite simple graph $ G $, the chromatic symmetric function $ X(G,\boldsymbol{x}) $ of $ G $ is the formal power series
\begin{align*}
X(G, \boldsymbol{x}) \coloneqq \sum_{f \in \Hom(G,K_{\mathbb{N}})} \prod_{v \in V_{G}}x_{f(v)}, 
\end{align*}
where $ \boldsymbol{x} $ denotes countably many commutative indeterminates $ x_{1}, x_{2}, \dots $ and $ K_{\mathbb{N}} $ the complete graph on positive integers $ \mathbb{N} $. 
The evaluation of $ X(G,\boldsymbol{x}) $ at $ 1^{n} = (\underbrace{1, \dots, 1}_{n \text{ times}}, 0, \dots) $ coincides with $\chi(G,n)$, the chromatic polynomial evaluated at $n$. 

Stanley found two distinct simple graphs with the same chromatic symmetric function. A natural problem is then to find a class of graphs which can be distinguished by the chromatic symmetric function. Stanley posed a conjecture in this direction. 
\begin{conjecture}[Stanley {\cite{stanley1995symmetric}}] \label{StaConj}
Suppose that $ T_{1} $ and $ T_{2} $ are finite trees and $ X(T_{1}, \boldsymbol{x}) = X(T_{2}, \boldsymbol{x}) $. 
Then $ T_{1} $ and $ T_{2} $ are isomorphic. 
\end{conjecture}
This conjecture is in contrast to the mentioned fact that trees on a fixed number of vertices have the same chromatic polynomial. This conjecture has not been solved, while several partial results are found in \cite{AlistePrietoZamora14,MartinMorinWagner08,OrellanaScott14,Smith:2015aa}. 

From now on, we consider a poset version of the above problem.  
For posets $ P $ and $ Q $, two kinds of homomorphisms may be considered. 
A map $ f \colon P \rightarrow Q $ is called a \textbf{strict} (resp.~\textbf{weak}) \textbf{homomorphism} if 
\begin{align*}
u < v \Rightarrow f(u) < f(v) \quad (\text{resp. }f(u) \leq f(v)). 
\end{align*}
Let $ \Hom^{<}(P,Q)\quad(\text{resp. }\Hom^{\leq}(P,Q)) $ denotes the set of strict (resp.~weak) homomorphisms from $ P $ to $ Q $. 
\begin{definition}
Let $ P $ be a finite poset. 
We define the \textbf{strict} (resp.~\textbf{weak}) \textbf{order quasisymmetric function} by 
\begin{align*}
\Gamma^{<}(P,\boldsymbol{x}) \coloneqq \sum_{f \in \Hom^{<}(P,\mathbb{N})} \prod_{v \in P}x_{f(v)} \hspace{5.5mm}\\
\left(\text{resp. }\Gamma^{\leq}(P,\boldsymbol{x}) \coloneqq \sum_{f \in \Hom^{\leq}(P,\mathbb{N})} \prod_{v \in P}x_{f(v)}\right). 
\end{align*}
\end{definition}
Note that these functions are kinds of $ (P,\omega) $-partition generating functions introduced by Stanley \cite[p.\ 81]{stanley1972ordered} and studied by Gessel \cite{gessel1984multipartite}. The evaluations $ \Gamma^{<}(P,1^{n}) $ and $ \Gamma^{\leq}(P,1^{n}) $ coincide with the order polynomials $ \overline{\Omega}(P,n) $ and $ \Omega(P,n) $ defined by Stanley \cite{stanley1970chromatic-like}. 
McNamara and Ward \cite{mcnamaca2014equality} also proved many properties of $ (P,\omega) $-partition generating functions as an invariant for finite labeled posets. In particular they gave two distinct finite posets which have the same order quasisymmetric function (see also Section \ref{open problem}). Thus $\Gamma^{<}(P,\boldsymbol{x})$ is not a complete invariant of the finite posets $P$. 
The reader is referred to \cite{gessel} for a historical survey. 

A \textbf{rooted tree} is a tree with a distinguished vertex called the \textbf{root}. 
Every vertex of a finite rooted tree $ R $ has a unique path from itself to the root. 
Hence $R$ is equipped with the natural order, i.e., $ u \leq v $ if the unique path from $ v $ to the root passes through $ u $. 
In this paper, we regard a rooted tree as a poset with respect to this order.

Our main result is that the order quasisymmetric functions distinguish finite rooted trees. 
\begin{theorem}\label{thm: main}
Let $ R_{1} $ and $ R_{2} $ be finite rooted trees. 
Then the following are equivalent.  
\begin{enumerate}[\rm(1)]
\item\label{thm: main1} $ \Gamma^{<}(R_{1},\boldsymbol{x}) = \Gamma^{<}(R_{2}, \boldsymbol{x}) $. 
\item\label{thm: main2} $ \Gamma^{\leq}(R_{1},\boldsymbol{x}) = \Gamma^{\leq}(R_{2}, \boldsymbol{x}) $. 
\item\label{thm: main3} $ R_{1} $ and $ R_{2} $ are isomorphic. 
\end{enumerate}
\end{theorem}
In fact we prove the result for a larger class which is characterized by the absence of full subposets ``$\N$'' and ``$\Join$'' (see Sections \ref{N-tie free posets} and \ref{proof of main theorem}).  The proof of our main theorem is based on algebraic properties of the ring of quasisymmetric functions.  
 
Our algebraic method may have potential applications to other invariants of posets, graphs or other objects. A possible future work is to try to find other invariants which can distinguish a class of graphs, posets or other objects, for instance the chromatic symmetric function and Tutte polynomial for graphs, and the W polynomial for weighted graphs \cite{NW99}. The reader can consult  \cite{EMJM11-1,EMJM11-2} for these and other invariants.  

The organization of this paper is as follows. 
In Section \ref{overlapping}, we introduce the overlapping shuffle algebra. 
In Section \ref{labeled posets}, we investigate properties of the strict order quasisymmetric functions with the theory of $ (P,\omega) $-partitions. We prove a key lemma (Lemma \ref{lem:min-max_irred}) about the irreducibility of $\Gamma^<(P,\boldsymbol{x})$. 
In Section \ref{N-tie free posets}, we introduce $ (\N,\Join) $-free posets and give their characterization. 
In Section \ref{proof of main theorem}, we give the proof of Theorem \ref{thm: main} for $ (\N,\Join) $-free posets. 
In Section \ref{open problem}, we propose some open problems. 

\section{The ring of quasisymmetric functions and the overlapping shuffle algebra}\label{overlapping}
A tuple $ (\alpha_{1}, \dots, \alpha_{\ell}) $ of positive integers is called a \textbf{composition}. 
Let $ \mathbb{N}^{\ast} $ denote the set of compositions (including the empty composition $ \varnothing $). 
A formal power series $ Q \in \mathbb{Z}[[\boldsymbol{x}]] $ is said to be \textbf{quasisymmetric} if the following conditions hold: 
\begin{enumerate}[\rm(1)]
\item The degree of $ Q $ is finite; 
\item The coefficients of the two monomials $ x_{1}^{\alpha_{1}} \cdots x_{\ell}^{\alpha_{\ell}} $ and $ x_{i_{1}}^{\alpha_{1}} \cdots x_{i_{\ell}}^{\alpha_{\ell}} $ in $Q$ are the same for any strictly increasing indices $ i_{1} < \cdots < i_{\ell} $ and any composition $ (\alpha_{1}, \dots, \alpha_{\ell}) $. 
\end{enumerate}
The set of quasisymmetric functions forms a subring of $ \mathbb{Z}[[\boldsymbol{x}]] $ (see Proposition \ref{prop:prodM} below) and is called the \textbf{ring of quasisymmetric functions}, which is denoted by $ \QSym $. 
For a composition $ \alpha = (\alpha_{1}, \dots, \alpha_{\ell}) $, we define the \textbf{monomial quasisymmetric function} $ M_{\alpha} $ by 
\begin{align*}
M_{\alpha} \coloneqq \sum_{i_{1} < \dots < i_{\ell}} x_{i_{1}}^{\alpha_{1}} \cdots x_{i_{\ell}}^{\alpha_{\ell}}, \qquad M_\varnothing :=1.  
\end{align*}
It is easy to show that the monomial quasisymmetric functions $ M_{\alpha} $ form a basis for $ \QSym $ as a module. 

For a positive integer $ \ell $, let $ [\ell] $ denote the totally ordered set $ \{1, \dots, \ell\} $ endowed with the usual order, and let $ [0] $ denote the empty set. 
\begin{proposition}[Hazewinkel {\cite[p.350]{hazewinkel1997leibniz}}, Grinberg-Reiner {\cite[Proposition 5.3]{grinberg2014hopf}}]\label{prop:prodM}
Let $ \alpha = (\alpha_{1}, \dots, \alpha_{\ell}) $ and $\beta=(\beta_{1}, \dots, \beta_{m})$ be compositions. Then 
\begin{align*}
M_{\alpha}M_{\beta} = \sum_{n=0}^{\infty} \sum_{(f,g) \in S(\ell,m,n)} M_{\wt_{\alpha, \beta}(f,g)}, 
\end{align*}
where 
\begin{align*}
S(\ell,m,n) \coloneqq \Set{(f,g) | \begin{array}{l}
f \colon [\ell] \rightarrow [n] \text{ and } g \colon [m] \rightarrow [n] \text{ are strictly} \\
\text{ order-preserving maps such that } f([\ell]) \cup g([m]) = [n]
\end{array}}, 
\end{align*}
and $ \wt_{\alpha, \beta}(f,g) = (\gamma_{1}, \dots, \gamma_{n}) $ denotes the composition of length $ n $ defined by 
\begin{align*}
\gamma_{k} \coloneqq \sum_{i \in f^{-1}(k)}\alpha_{i} + \sum_{j \in g^{-1}(k)}\beta_{j} \text{ for each } k \in [n]. 
\end{align*}
Note that $ n $ actually runs from $ \max\{\ell, m\} $ to $ \ell + m $. 
\end{proposition}

\begin{example}
We represent a pair $ (f,g) \in S(\ell,m,n) $ as boxes arranged in $ 2 $ rows and $ n $ columns with numbers placed in some boxes. 
When $ f(i)=k $, we place $ i $ in the first row and $ k $-th column. 
Additionally, if $ g(j)=k $, we place $ j $ in the second row and $ k $-th column. 
For instance, 
\begin{align*}
\begin{ytableau}
1 &   & 2 & 3 \\
 & 1 & & 2 
\end{ytableau}
\end{align*}
denotes the pair $ (f,g) \in S(3,2,4) $, where $ f(1)=1,f(2)=3,f(3)=4 $ and $g(1)=2, g(2)=4 $. 
The corresponding composition $ \wt_{\alpha, \beta}(f,g) $ is equal to $ (\alpha_{1}, \beta_{1}, \alpha_{2}, \alpha_{3}+\beta_{2}) $. 

We calculate $ M_{(\alpha_{1},\alpha_{2})}M_{(\beta_{1})} $ by using Proposition \ref{prop:prodM}. 
We have 
\begin{align*}
S(2,1,2) &= \Set{
\begin{ytableau}
1 & 2 \\
1 & 
\end{ytableau} \ , \
\begin{ytableau} 
1 & 2 \\
 & 1
\end{ytableau}
}, \\
S(2,1,3) &=\Set{
\begin{ytableau} 
1 & 2 & \\
  & & 1 
\end{ytableau} \ , \
\begin{ytableau} 
1 & & 2 \\
 & 1 & 
\end{ytableau} \ , \
\begin{ytableau} 
\  &1 & 2 \\
1 &  & 
\end{ytableau}
}. 
\end{align*}
Hence we obtain 
\begin{align*}
M_{(\alpha_{1},\alpha_{2})}M_{(\beta_{1})}
=M_{(\alpha_{1}+\beta_{1}, \alpha_{2})}
+M_{(\alpha_{1}, \alpha_{2}+\beta_{1})}
+M_{(\alpha_{1}, \alpha_{2},\beta_{1})}
+M_{(\alpha_{1}, \beta_{1}, \alpha_{2})}
+M_{(\beta_{1}, \alpha_{1}, \alpha_{2})}. 
\end{align*}
\end{example}

Let $ \mathcal{M} \coloneqq \bigoplus_{\alpha \in \mathbb{N}^{\ast}} \mathbb{Z}\alpha $. 
The map $ M \colon \mathcal{M} \rightarrow \QSym $, defined by the linear extension of $ \alpha \mapsto M_{\alpha} $, is an isomorphism of modules. 
We introduce two products on $ \mathcal{M} $. 
One is the noncommutative product $\ast$ called the \textbf{concatenation}, which is the linear extension of concatenation of compositions $ \alpha \ast \beta \coloneqq (\alpha_{1}, \dots, \alpha_{\ell}, \beta_{1}, \dots, \beta_{m}) $ for $ \alpha = (\alpha_{1}, \dots, \alpha_{\ell}), \beta=(\beta_{1}, \dots, \beta_{m}) \in \mathbb{N}^{\ast} $. 
The algebra $ (\mathcal{M}, \ast) $ is isomorphic to the free algebra $ \mathbb{Z}\langle\mathbb{N}\rangle $. 
The other is a commutative product defined so that the map $ M $ becomes an isomorphism of algebras from $ \mathcal{M} $ to $ \QSym $. 
This product is called the \textbf{overlapping shuffle product} and denoted by $ \oshuffle $. 
The algebra $ (\mathcal{M}, \oshuffle) $ is called the \textbf{overlapping shuffle algebra}. It has the unit given by the empty composition. 
The module isomorphism $ M \colon \mathcal{M} \rightarrow \QSym $ also induces the concatenation product $\ast$ on $\QSym$, which is the linear extension of $M_{\alpha} \ast M_{\beta} := M_{\alpha \ast\beta}$. The algebra $(\QSym,\ast)$ is noncommutative. 

\begin{theorem}[Hazewinkel {\cite[Theorem 8.1]{hazewinkel2001algebra}}]\label{thm:Hazewinkel}
The ring of quasisymmetric functions $ \QSym $ and the overlapping shuffle algebra $ (\mathcal{M}, \oshuffle) $ are free commutative algebras. 
\end{theorem}
Note that Hazewinkel gave explicit generators for $ (\mathcal{M}, \oshuffle) $ but we only require the following corollary in this paper. 
\begin{corollary}\label{cor:UFD}
The ring of quasisymmetric functions $ \QSym $ is a unique factorization domain. 
\end{corollary}
Before Hazewinkel proved Theorem \ref{thm:Hazewinkel}, 
Malvenuto \cite[Corollary 4.19]{malvenuto1993produits} proved that the algebra $ \QSym \otimes_{\mathbb{Z}} \mathbb{Q} $ is a free commutative algebra. 
For our purpose, we may adopt this weaker theorem. 

There is a recurrence formula for the overlapping shuffle product, which is useful for computing and investigating it. 
\begin{proposition}\label{prop:osp_rec}
For non-empty compositions $ \alpha=(\alpha_{1}, \dots, \alpha_{\ell}) $ and $ \beta=(\beta_{1}, \dots, \beta_{m}) $, we have that 
\begin{align*}
\alpha \oshuffle \beta = (\alpha_{1}) \ast (\alpha^{\prime} \oshuffle \beta) + (\beta_{1}) \ast (\alpha \oshuffle \beta^{\prime}) + (\alpha_{1}+\beta_{1}) \ast (\alpha^{\prime} \oshuffle \beta^{\prime}), 
\end{align*}
where $ \alpha^{\prime}, \beta^{\prime} $ are the compositions satisfying $ \alpha=(\alpha_{1})\ast \alpha^{\prime} $ and $ \beta=(\beta_{1})\ast \beta^{\prime} $. 
\end{proposition}
\begin{proof}
From Proposition \ref{prop:prodM}, we have that 
\begin{align*}
\alpha \oshuffle \beta = \sum_{n=0}^{\infty}\sum_{(f,g) \in S(\ell,m,n)}\wt_{\alpha, \beta}(f,g). 
\end{align*}
Define subsets of $ S(\ell,m,n) $ by 
\begin{align*}
S_{1} &\coloneqq \Set{(f,g) \in S(\ell,m,n) | f(1)=1 \text{ and } g(1)>1}, \\
S_{2} &\coloneqq \Set{(f,g) \in S(\ell,m,n) | f(1)>1 \text{ and } g(1)=1}, \\
S_{3} &\coloneqq \Set{(f,g) \in S(\ell,m,n) | f(1)=1 \text{ and } g(1)=1}. 
\end{align*}
Then $ S(\ell, m ,n) = S_{1} \sqcup S_{2} \sqcup S_{3} $. 
For each pair $ (f,g) \in S_{1} $, define a pair $ (f^{\prime},g^{\prime}) \in S(\ell-1,m,n-1) $ by 
\begin{align*}
f^{\prime}(i) \coloneqq f(i+1)-1 \text{ for } i \in [\ell-1]
\text{ and } g^{\prime}(j) \coloneqq g(j)-1 \text{ for } j \in [m]. 
\end{align*}
This correspondence is bijective and we have that $ \wt_{\alpha, \beta}(f,g) = (\alpha_{1}) \ast \wt_{\alpha^{\prime}, \beta}(f^{\prime},g^{\prime}) $. 
Hence 
\begin{align*}
\sum_{n=0}^{\infty}\sum_{(f,g) \in S_{1}} \wt_{\alpha, \beta}(f,g) 
= \sum_{n=0}^{\infty}\sum_{(f^{\prime},g^{\prime}) \in S(\ell-1,m,n-1)}(\alpha_{1}) \ast \wt_{\alpha^{\prime},\beta}(f^{\prime},g^{\prime})
= (\alpha_{1}) \ast (\alpha^{\prime} \oshuffle \beta). 
\end{align*}
A similar discussion yields that 
\begin{align*}
\sum_{n=0}^{\infty}\sum_{(f,g) \in S_{2}}\wt_{\alpha,\beta}(f,g) = (\beta_{1}) \ast (\alpha \oshuffle \beta^{\prime}) 
\text{ and } 
\sum_{n=0}^{\infty}\sum_{(f,g) \in S_{3}}\wt_{\alpha,\beta}(f,g) = (\alpha_{1} + \beta_{1}) \ast (\alpha^{\prime} \oshuffle \beta^{\prime}). 
\end{align*}
Therefore the assertion holds. 
\end{proof}

\begin{example} Proposition \ref{prop:osp_rec} shows that 
\begin{align*}
(\alpha_1) \oshuffle (\beta_1) &= (\alpha_1,\beta_1) + (\beta_1,\alpha_1) + (\alpha_1+\beta_1), \\
(\alpha_1,\alpha_2) \oshuffle (\beta_1) &= (\alpha_1) \ast((\alpha_2)\oshuffle (\beta_1)) + (\beta_1) \ast (\alpha_1, \alpha_2) + (\alpha_1+\beta_1) \ast (\alpha_2) \\
 &= (\alpha_1,\alpha_2,\beta_1) + (\alpha_1,\beta_1,\alpha_2) + (\alpha_1,\alpha_2+\beta_1) + (\beta_1, \alpha_1,\alpha_2) + (\alpha_1+\beta_1, \alpha_2). 
\end{align*}
\end{example}

\begin{definition}
We introduce the \textbf{lexicographical order} $ \leq $ on $ \mathbb{N}^{\ast} $, i.e., for compositions $ \alpha = (\alpha_{1}, \dots, \alpha_{\ell}), \beta = (\beta_{1}, \dots, \beta_{m}) $, we denote by $ \alpha < \beta $ if one of the following conditions holds. 
\begin{enumerate}[\rm(1)]
\item $\alpha=\varnothing$ and $\beta \neq \varnothing$. 
\item There exists $ i \in \Set{1, \dots, \max\{\ell,m\}} $ such that $ \alpha_{1}=\beta_{1}, \dots, \alpha_{i-1}=\beta_{i-1} $ and $ \alpha_{i} < \beta_{i} $. 
\item $ \ell<m $ and $ \alpha_{1}=\beta_{1}, \dots, \alpha_{\ell} = \beta_{\ell} $. 
\end{enumerate}
\end{definition}
\begin{definition}
The \textbf{leading term} $ \LT(q) $ of an element $ q \in \mathcal{M} $ is the term which contains the greatest composition with respect to the lexicographical order.  %The \textbf{leading coefficient} of $q$ is the coefficient of the leading term.  
\end{definition}

\begin{definition}
For compositions $ \alpha=(\alpha_{1}, \dots, \alpha_{\ell}) $ and $ \beta=(\beta_{1}, \dots, \beta_{m}) $, we define a composition $ \alpha \dotplus \beta $ by the coordinatewise sum, i.e., 
\begin{align*}
\alpha \dotplus \beta \coloneqq 
\begin{cases}
(\alpha_{1}+\beta_{1}, \dots, \alpha_{\ell}+\beta_{\ell}, \beta_{\ell+1}, \dots, \beta_{m}) & \text{ if }\ell \leq m, \\
(\alpha_{1}+\beta_{1}, \dots, \alpha_{m}+\beta_{m}, \alpha_{m+1}, \dots, \alpha_{\ell}) & \text{ if }\ell \geq m. 
\end{cases}
\end{align*}
\end{definition}

\begin{proposition}\label{prop:LT1}
Let $ \alpha $ and $ \beta $ be compositions. 
Then the leading term of $ \alpha \oshuffle \beta $ is $ \alpha \dotplus \beta $. 
\end{proposition}
\begin{proof}
When $ \alpha $ or $ \beta $ is empty, then the assertion is obvious. 
Suppose that both $ \alpha $ and $ \beta $ are non-empty. 
By Proposition \ref{prop:osp_rec}, we have that 
\begin{align*}
\LT(\alpha \oshuffle \beta) = (\alpha_{1}+\beta_{1}) \ast \LT(\alpha^{\prime} \oshuffle \beta^{\prime}), 
\end{align*}
where $ \alpha^{\prime}, \beta^{\prime} $ denote compositions satisfying $ \alpha=(\alpha_{1}) \ast \alpha^{\prime}, \beta=(\beta_{1})\ast\beta^{\prime} $. 
Using induction, we have that $ \LT(\alpha^{\prime} \oshuffle \beta^{\prime}) = \alpha^{\prime} \dotplus \beta^{\prime} $. 
Hence $ \LT(\alpha \oshuffle \beta) = (\alpha_{1}+\beta_{1}) \ast (\alpha^{\prime} \dotplus \beta^{\prime}) = \alpha \dotplus \beta $. 
\end{proof}

\begin{proposition}\label{prop:LT2}
Let $ p,q \in \mathcal{M} $ and $ \LT(p) = c\alpha, \LT(q) = d\beta $. 
Then $ \LT(p \oshuffle q) = c d(\alpha \dotplus \beta) $. 
\end{proposition}
\begin{proof}
Write $ p = \sum_{\gamma \leq \alpha} c_{\gamma}\gamma, q = \sum_{\delta \leq \beta} d_{\delta}\delta $, and $ p \oshuffle q = \sum_{\substack{\gamma \leq \alpha\\ \delta \leq \beta}} c_{\gamma}d_{\delta}(\gamma \oshuffle \delta) $. 
If $ \gamma < \alpha $ or $ \delta < \beta $, then $ \gamma \dotplus \delta < \alpha \dotplus \beta $. 
From Proposition \ref{prop:LT1}, the composition $ \alpha \dotplus \beta $ is the greatest composition in $ p \oshuffle q $. 
Therefore the assertion holds. 
\end{proof}

\begin{definition}
For a composition $ \alpha=(\alpha_{1}, \dots, \alpha_{\ell}) $, we define the \textbf{reverse} of $ \alpha $ by $ \alpha^{r} \coloneqq (\alpha_{\ell}, \dots, \alpha_{1}) $. 
Define an involution $ \rho $ on $ \mathcal{M} $ by the linear extension of the reverse. 
\end{definition}

\begin{proposition}\label{prop:rev_osp_comp}
The map $ \rho $ is compatible with the overlapping shuffle product, i.e., $ \rho(\alpha \oshuffle \beta) = \alpha^{r} \oshuffle \beta^{r} $. 
\end{proposition}
\begin{proof}
For each pair $ (f,g) \in S(\ell,m,n) $, define a pair $ (f^{\prime},g^{\prime}) \in S(\ell,m,n) $ by 
\begin{align*}
f^{\prime}(i) &\coloneqq n+1 - f(\ell+1-i), \text{ where } i, \ell+1-i \in [\ell], \\
\text{ and } g^{\prime}(j) &\coloneqq n+1 - g(m+1-j) , \text{ where } j, m+1-j \in [m].  
\end{align*}
This correspondence is bijective and we have that $ \wt_{\alpha, \beta}(f,g)^{r} = \wt_{\alpha^{r}, \beta^{r}}(f^{\prime},g^{\prime}) $. 
Hence, by Proposition \ref{prop:prodM}, 
\begin{align*}
(\alpha \oshuffle \beta)^{r}
= \sum_{n=0}^{\infty}\sum_{(f,g) \in S(\ell,m,n)}\wt_{\alpha, \beta}(f,g)^{r}
= \sum_{n=0}^{\infty}\sum_{(f^{\prime},g^{\prime}) \in S(\ell,m,n)}\wt_{\alpha^{r},\beta^{r}}(f^{\prime},g^{\prime})
= \alpha^{r} \oshuffle \beta^{r}. 
\end{align*}
\end{proof}

\begin{definition}
An element in $ \mathcal{M} $ is said to be \textbf{primitive} if the greatest common divisor of its coefficients is $ 1 $. 
\end{definition}

\begin{lemma}\label{lem:irred}
Let $ q \in \mathcal{M} $ be a primitive nonzero element. 
Then $ (1) \ast q $ and $ q \ast (1) $ are irreducible with respect to the overlapping shuffle product. 
\end{lemma}
\begin{proof}
By Proposition \ref{prop:rev_osp_comp}, it suffices to show irreducibility of $ (1)\ast q $. 
Assume that $ (1)\ast q $ is reducible. 
Since $ (1)\ast q $ is primitive, there exist non-constant elements $ p,p^{\prime} \in \mathcal{M} $ such that $ (1)\ast q = p \oshuffle p^{\prime} $. 
When $ \LT(p)=c\alpha, \LT(p^{\prime}) = d\beta $, 
the leading term of $ p \oshuffle p^{\prime} $ is $ cd(\alpha \dotplus \beta) $ by Proposition \ref{prop:LT2}. 
Since $ \alpha $ and $ \beta $ are non-empty, the first component of $ \alpha \dotplus \beta $ is greater than $ 1 $. 
However, the first component of each term of $ (1)\ast q $ is $ 1 $, which is a contradiction. 
\end{proof}

\section{Labeled posets and their quasisymmetric generating functions}\label{labeled posets}

A labeling of a finite poset $ P $ is a bijection $ \omega \colon P \rightarrow [|P|] $, where $ |P| $ denotes its cardinality. 
The pair $ (P, \omega) $ is called a \textbf{labeled poset}. 
A labeling $ \omega $ is said to be \textbf{strict} (resp.~\textbf{natural}) if 
\begin{align*}
u < v &\Rightarrow \omega(u) > \omega(v) \\
(\text{resp. } u < v &\Rightarrow \omega(u) < \omega(v)). 
\end{align*}
A \textbf{$ (P,\omega) $-partition} is a map $ f \colon P \rightarrow \mathbb{N} $ satisfying the following two conditions: 
\begin{enumerate}[\rm(1)]
\item $ u \leq v $ implies $ f(u) \leq f(v) $; 
\item $ u \leq v $ and $ \omega(u) > \omega(v) $ imply $ f(u) < f(v) $. 
\end{enumerate}
Let $ A(P,\omega) $ denote the set of $ (P,\omega) $-partitions. 
Note that if $ \omega $ is strict (resp.~natural) then $ A(P,\omega) $ coincides with $ \Hom^{<}(P,\mathbb{N}) $ (resp.~$ \Hom^{\leq}(P,\mathbb{N}) $). 

\begin{definition}
For a labeled poset $ (P,\omega) $,  the \textbf{$ (P,\omega) $-partition generating function} is the formal power series
\begin{align*}
\Gamma(P,\omega, \boldsymbol{x}) \coloneqq \sum_{f \in A(P,\omega)}\prod_{v \in P}x_{f(v)}. 
\end{align*}
\end{definition} 
It is easy to show that $ \Gamma(P,\omega,\boldsymbol{x}) $ is a quasisymmetric function (see also Proposition \ref{prop:expM}). 
Note that if the labeling $ \omega $ is strict (resp.~natural) then $ \Gamma(P,\omega,\boldsymbol{x}) $ coincides with $ \Gamma^{<}(P,\boldsymbol{x}) $ (resp.~$ \Gamma^{\leq}(P,\boldsymbol{x}) $). The one-point poset $[1]$ has the $ (P,\omega) $-partition generating function $\Gamma([1],\omega, \boldsymbol{x}) =M_{(1)}$. See Example \ref{exa:SOP} for other examples. 
A \textbf{complementary labeling} $ \overline{\omega} $ of $ \omega $ is defined by $ \overline{\omega}(v) \coloneqq |P|+1-\omega(v) $. 
If $ \omega $ is strict (resp.~natural) then $ \overline{\omega} $ is natural (resp.~strict). 

\begin{proposition}[McNamara-Ward {\cite[Proposition 3.7]{mcnamaca2014equality}}]\label{prop:st_wk}
Let $ (P,\omega) $ and $ (Q,\tau) $ be labeled posets. 
Then $ \Gamma(P,\omega,\boldsymbol{x}) = \Gamma(Q,\tau, \boldsymbol{x}) $ if and only if $ \Gamma(P,\overline{\omega},\boldsymbol{x}) = \Gamma(Q,\overline{\tau}, \boldsymbol{x}) $. 
In particular, $ \Gamma^{<}(P, \boldsymbol{x}) = \Gamma^{<}(Q, \boldsymbol{x}) $ if and only if $ \Gamma^{\leq}(P, \boldsymbol{x}) = \Gamma^{\leq}(Q, \boldsymbol{x}) $. 
\end{proposition}
Proposition \ref{prop:st_wk} shows the equivalence between \eqref{thm: main1} and \eqref{thm: main2} in Theorem \ref{thm: main}. 
Hence we may focus on the strict order quasisymmetric function $ \Gamma^{<}(P,\boldsymbol{x}) $. 

\begin{proposition}[McNamara-Ward {\cite[Proposition 3.7 and Corollary 4.3]{mcnamaca2014equality}}]\label{prop:number_min_max}
Let $ P $ and $ Q $ be finite posets such that $ \Gamma^{<}(P, \boldsymbol{x}) = \Gamma^{<}(Q,\boldsymbol{x}) $. 
Then $ P $ and $ Q $ have the same number of minimal elements, and the same number of maximal elements. 
\end{proposition}

\begin{definition}\label{SOP}
A \textbf{stable ordered partition} of a labeled poset $ (P,\omega) $ is a tuple $ \Pi = (\pi_{1}, \dots, \pi_{\ell}) $ consisting of non-empty subsets of $ P $ which satisfy the following conditions: 
\begin{enumerate}[\rm(1)]
\item $ P = \sqcup_{i=1}^{\ell}\pi_{i} $; 
\item $ u \in \pi_{i}, v \in \pi_{j} $ and $ u \leq v $ imply $ i \leq j $; 
\item $ u \in \pi_{i}, v \in \pi_{j}, u < v $ and $ \omega(u) > \omega(v) $ imply $ i < j $. 
\end{enumerate}
Let $ \St(P,\omega) $ denote the set of the stable ordered partitions of $ (P,\omega) $. 
The \textbf{type} of a stable ordered partition $ \Pi = (\pi_{1}, \dots, \pi_{\ell}) $ is the composition $ (|\pi_{1}|, \dots, |\pi_{\ell}|) $ and is denoted by $ \type(\Pi) $. The set of stable ordered partitions of type $ \alpha $ is denoted by $ \St_{\alpha}(P,\omega) $. 
When the labeling $ \omega $ is strict, the sets $ \St(P,\omega) $ and $ \St_{\alpha}(P,\omega) $ are denoted by $ \St^{<}(P) $ and $ \St_{\alpha}^{<}(P) $, respectively. Note that $\St_\alpha^<(P)=\varnothing$ unless $\sum_{i=1}^\ell \alpha_i=|P|$ when $\alpha=(\alpha_1,\dots, \alpha_\ell)$. 
\end{definition}

The expansion of $ (P,\omega) $-partition generating functions pointed out by McNamara and Ward \cite[p.493]{mcnamaca2014equality} reads as follows in terms of stable ordered partitions. 
\begin{proposition}\label{prop:expM}
Let $ (P,\omega) $ be a labeled poset. 
Then 
\begin{align*}
\Gamma(P,\omega,\boldsymbol{x}) = \sum_{\alpha \in \mathbb{N}^{\ast}} |\St_{\alpha}(P,\omega)|M_{\alpha}. 
\end{align*}
In particular, 
\begin{align*}
\Gamma^{<}(P,\boldsymbol{x}) = \sum_{\alpha \in \mathbb{N}^{\ast}}|\St_{\alpha}^{<}(P)|M_{\alpha}. 
\end{align*}
\end{proposition}
\begin{proof}
For each $ f \in A(P,\omega) $ there is a sequence of increasing indices $ i_{1} < \dots < i_{\ell} $ such that $ \{i_{1}, \dots, i_{\ell}\} = \Set{i \in \mathbb{N} | f^{-1}(i) \neq \varnothing} $. 
Define a stable ordered partition corresponding to $ f $ by $ \Pi_{f} \coloneqq (f^{-1}(i_{1}), \dots, f^{-1}(i_{\ell})) $. 
The map $ A(P,\omega) \rightarrow \St(P,\omega); f \mapsto \Pi_{f} $ is surjective. 
We define an equivalence relation $ f \sim g $ on $ A(P,\omega) $ by $ \Pi_{f} = \Pi_{g} $. 
Then for each $ f \in A(P,\omega) $, we have that 
\begin{align*}
\sum_{g \sim f}\prod_{v \in P}x_{g(v)} = M_{\type(\Pi_{f})}. 
\end{align*}
Therefore the desired result follows. 
\end{proof}

\begin{example}\label{exa:SOP}
\begin{figure}
\centering
\begin{tikzpicture}
\draw (0,0) node[v](1){};
\draw (-1,1) node[v](2){};
\draw (1,1) node[v](3){};
\draw (1)--(2); 
\draw (1)--(3);
\end{tikzpicture}
\caption{Hasse diagram of $\vee=\{a,b,c\}$}\label{fig:SOP}
\end{figure}
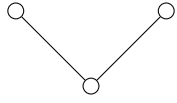

Let $\vee=\{a,b,c\}$ be the poset with the two relations $a<b$ and $a<c$ in Fig.\ \ref{fig:SOP}, and let $\omega$ be the strict labeling. Then
the second condition in Definition \ref{SOP} is contained in the third, and the conditions become ``$u \in \pi_{i}, v \in \pi_{j}, u < v $ imply $ i < j $''. The ordered stable partitions are listed below: 
\[
(\{a\},\{b\},\{c\}), (\{a\},\{c\},\{b\}), (\{a\},\{b,c\}). 
\]
Then Proposition \ref{prop:expM} tells us that
\[
\Gamma^<(P,\boldsymbol{x}) = \sum_{\substack{i<j\\ i<k}}x_i x_j x_k = 2 M_{(1,1,1)} + M_{(1,2)}. 
\]

Another example is the same poset $\vee=\{a,b,c\}$ with the different labeling $\omega'(a)=2, \omega'(b)=1, \omega'(c)= 3$. In this case the third condition in Definition \ref{SOP} forbids the elements $a$ and $b$ to belong to the same block. The ordered stable partitions are given by
\[
(\{a\},\{b\},\{c\}), (\{a\},\{c\},\{b\}), (\{a\},\{b,c\}), (\{a,c\},\{b\}) 
\]
and so 
\[
\Gamma(P,\omega',\boldsymbol{x}) =  \sum_{\substack{i<j\\ i\leq k}}x_i x_j x_k=2 M_{(1,1,1)} + M_{(1,2)} +M_{(2,1)}. 
\]
\end{example}

For elements $ u,v $ in a poset $ P $, we say that $ v $ \textbf{covers} $ u $ if $ u < v $ and there are no elements $ w \in P $ such that $ u < w < v $. 
If $v$ covers $ u $ then the pair $ (u,v) $ is an edge of the Hasse diagram of $ P $. 
For a labeled poset $ (P,\omega) $, we say that $ (u,v) $ is a \textbf{strict edge} if $v$ covers $u$ and $ \omega(u) > \omega(v) $. 

\begin{definition}
The \textbf{jump} of an element $ v $ in a labeled poset $ (P,\omega) $ is the maximum number of strict edges in saturated chains from $ v $ down to a minimal element in $ P $. 
The \textbf{jump sequence} of $ P $, denoted by $ \jump(P,\omega) $, is a composition $ (j_{0}, \dots, j_{\ell}) $, where $ j_{i} $ denotes the number of elements with jump $ i $. 
\end{definition}

We introduce the lexicographical order on the monomial quasisymmetric functions $ M_{\alpha} $, i.e., $ M_{\alpha} \leq M_{\beta} \Leftrightarrow \alpha \leq \beta $. Then the leading term of a quasisymmetric function $ Q=\sum_{\alpha \in \mathbb{N}^{\ast}}c_{\alpha}M_{\alpha} $ is the term $ c_{\alpha}M_{\alpha} $, where $ M_{\alpha} $ is the maximum monomial quasisymmetric function in $ Q $ with $ c_{\alpha} \neq 0 $. In this case $c_\alpha$ is called the leading coefficient. 

\begin{proposition}[McNamara-Ward {\cite[Proposition 4.2]{mcnamaca2014equality}}]\label{prop:lead_coeff}
For a labeled poset $ (P,\omega) $, the leading term of $ \Gamma(P,\omega,\boldsymbol{x}) $ is $ M_{\jump(P,\omega)} $. 
In particular, the leading coefficient of $ \Gamma(P,\omega,\boldsymbol{x}) $ is $ 1 $ and $ \Gamma(P,\omega,\boldsymbol{x}) $ is primitive. 
\end{proposition}

\begin{definition}
For posets $ P $ and $ Q $, the \textbf{disjoint union} $ P \sqcup Q $ is a poset whose underlying set is the disjoint union of $ P $ and $ Q $ and the order $ \leq $ is defined in such a way that $ u \leq v $ if and only if 
\begin{enumerate}[(1)]
\item $ u,v \in P $ and $ u \leq_{P} v $, or 
\item $ u,v \in Q $ and $ u \leq_{Q} v $. 
\end{enumerate}
\end{definition}
\begin{remark}
The disjoint union is also called the \textbf{parallel composition}. 
\end{remark}

\begin{definition}
For posets $ P $ and $ Q $, the \textbf{ordinal sum} $ P \oplus Q $ is a poset whose underlying set is the disjoint union of $ P $ and $ Q $, and the order $\leq$ is defined in such a way that $u \leq v$ if and only if  
\begin{enumerate}[\rm(1)]
\item $ u,v \in P $ and $ u \leq_P v $, or 
\item $ u,v \in Q $ and $ u \leq_Q v $, or
\item $ u \in P $ and $ v \in Q $. 
\end{enumerate}
\end{definition}
\begin{remark}
The ordinal sum is also called the \textbf{series composition} or the \textbf{linear sum}. 
\end{remark}

\begin{proposition}\label{prop:sstp_ordsum}
Let $ P $ and $ Q $ be finite posets.
Then the map $ \phi \colon \St^{<}(P) \times \St^{<}(Q) \rightarrow \St^{<}(P \oplus Q)$ defined by $ \phi(\Pi,\Pi^{\prime}) \coloneqq (\Pi,\Pi^{\prime})$ is a bijection. In particular, it induces the bijection 
$$
\bigcup_{\substack{\alpha,\beta \in \mathbb{N}^\ast \\ \alpha\ast \beta =\gamma}}\left(\St^{<}_{\alpha}(P) \times \St^{<}_{\beta}(Q) \right) \simeq \St^{<}_{\gamma}(P\oplus Q)
$$ 
for any composition $\gamma$. 
\end{proposition}
\begin{proof}
The injectivity is clear. 
To show the surjectivity, take $ \tilde{\Pi} = (\pi_{1}, \dots, \pi_{\ell}) \in \St^{<}(P \oplus Q) $. 
Suppose that $ u \in P, v \in Q $ and $ u \in \pi_{i}, v \in \pi_{j} $. 
By the definition of the ordinal sum, we have that $ u < v $. 
The definition of a stable ordered partition implies that $ i < j $. 
Therefore there exists an index $ k $ such that $ P = \sqcup_{i=1}^{k}\pi_{i} $ and $ Q=\sqcup_{i=k+1}^{\ell}\pi_{i} $. 
Taking $ \Pi \coloneqq (\pi_{1}, \dots, \pi_{k}) $ and $ \Pi^{\prime} \coloneqq (\pi_{k+1}, \dots, \pi_{\ell}) $, the surjectivity of $ \phi $ follows. The last identity holds since the index $k$ runs over $[\ell]$ depending on $ \tilde{\Pi} \in \St^{<}_\gamma(P \oplus Q) $. 
\end{proof}

Recall that the concatenation $ \ast $ on $ \QSym $ is the linear extension of $ M_{\alpha} \ast M_{\beta} = M_{\alpha \ast \beta} $ (see Theorem \ref{thm:Hazewinkel}). 
\begin{proposition}\label{lem:prod}
Let $ P $ and $ Q $ be finite posets. 
Then the following assertions hold. 
\begin{enumerate}[\rm(a)]
\item\label{lem:prod1} $ \Gamma^{<}(P \sqcup Q, \boldsymbol{x}) = \Gamma^{<}(P,\boldsymbol{x})\Gamma^{<}(Q,\boldsymbol{x}) $. 
\item\label{lem:prod2} $ \Gamma^{<}(P \oplus Q, \boldsymbol{x}) = \Gamma^{<}(P, \boldsymbol{x}) \ast \Gamma(Q, \boldsymbol{x}) $. 
\end{enumerate}
\end{proposition} 
\begin{proof}
\eqref{lem:prod1} is due to Malvenuto \cite[Proposition 4.6]{malvenuto1993produits} or McNamara and Ward \cite[Proposition 3.4]{mcnamaca2014equality}. 

\eqref{lem:prod2} By Proposition \ref{prop:expM} and Proposition \ref{prop:sstp_ordsum}, we have that 
\begin{align*}
\Gamma^{<}(P\oplus Q, \boldsymbol{x})
&= \sum_{\gamma \in \mathbb{N}^{\ast}} |\St^<_{\gamma}(P \oplus Q)| M_{\gamma} \\
&= \sum_{\gamma \in \mathbb{N}^{\ast}} \sum_{\substack{\alpha,\beta \in \mathbb{N}^\ast \\ \alpha\ast \beta =\gamma}}|\St^<_{\alpha}(P)||\St^<_{\beta}(Q)|M_{\alpha} \ast M_{\beta} \\
&=\sum_{\alpha,\beta \in \mathbb{N}^\ast}|\St^<_{\alpha}(P)||\St^<_{\beta}(Q)|M_{\alpha} \ast M_{\beta} \\
&= \left(\sum_{\alpha \in \mathbb{N}^\ast}|\St^<_{\alpha}(P)|M_{\alpha}\right)\ast \left(\sum_{\beta \in \mathbb{N}^{\ast}}|\St^<_{\beta}(Q)|M_{\beta}\right) \\
&= \Gamma^{<}(P, \boldsymbol{x}) \ast \Gamma^{<}(Q,\boldsymbol{x}),  
\end{align*}
the conclusion. 
\end{proof}

%The next result gives a partial answer to a problem posed by McNamara-Ward \cite[Question 7.2]{mcnamaca2014equality}. Recall that an element $u$ of a poset $P$ is called a \textbf{maximum} if $u \leq v$ for all $v\in P$. A minimum element is defined similarly. A maximum  exists if and only if $P$ has a unique maximal element. 

We prove the irreducibility of strict order quasisymmetric functions, which is a crucial ingredient in the proof of the main theorem (see Theorem \ref{thm:unique_C}).

\begin{lemma}\label{lem:min-max_irred}
Suppose that a finite poset $ P $ has a unique minimal or maximal element. 
Then the strict order quasisymmetric function $ \Gamma^{<}(P,\boldsymbol{x}) $ is irreducible in $ \QSym $. 
\end{lemma}
\begin{proof}
By assumption, $ P $ is of the form $ [1]\oplus P' $ or $ P' \oplus [1] $. 
By Proposition \ref{lem:prod} \eqref{lem:prod2}, we have that $ \Gamma^{<}(P,\boldsymbol{x}) $ is equal to $ M_{(1)} \ast \Gamma^{<}(P', \boldsymbol{x}) $ or $ \Gamma^{<}(P', \boldsymbol{x}) \ast M_{(1)} $, respectively. 
By Proposition \ref{prop:lead_coeff}, $ \Gamma^{<}(P',\boldsymbol{x}) $ is primitive. 
Hence Lemma \ref{lem:irred} forces $ \Gamma^{<}(P,\boldsymbol{x}) $ to be irreducible in $ \QSym $ in both cases. 
\end{proof}

\section{$ (\N,\Join) $-free posets}\label{N-tie free posets}
This section introduces a concept of $(\N,\Join)$-free posets, which is the class to be considered in our main theorem (see Theorem \ref{thm:unique_C}). 
A \textbf{full subposet} (or an \textbf{induced subposet}) of a poset $ P $ is a subset of the underlying set of $ P $ equipped with the induced order. 
Let $\N$ be the poset consisting of four elements $a,b,c,d$ endowed with the three relations $a< b > c < d$ (see Fig.\ \ref{fig:N-tie}). A poset is \textbf{$\N$-free} if it does not contain a full subposet that is isomorphic to $\N$.
\begin{remark} Some authors use the term ``$\N$-free'' with different meanings. Our use follows the book \cite{CLM12}, not \cite{NK98}.  
\end{remark}

\begin{figure}
\centering
\begin{tikzpicture}
\draw (0,0) node[v](1){};
\draw (0,1) node[v](2){};
\draw (1,0) node[v](3){};
\draw (1,1) node[v](4){};
\draw (1)--(2)--(3)--(4);
\end{tikzpicture}
\qquad
\begin{tikzpicture}
\draw (0,0) node[v](1){};
\draw (0,1) node[v](2){};
\draw (1,0) node[v](3){};
\draw (1,1) node[v](4){};
\draw (1)--(2)--(3)--(4);
\draw (1)--(4); 
\end{tikzpicture}
\caption{Hasse diagrams of $ \N $ and $ \Join $}\label{fig:N-tie}
\end{figure}

Let $ \Join $ be the poset consisting four elements $ a,b,c,d $ with the four relations $ a<b>c<d $ and $ a<d $ (see Fig.\ \ref{fig:N-tie}). 
\begin{definition}
A poset is \textbf{$ (\N,\Join) $-free} if it does not contain a full subposet that is isomorphic to $ \N $ or $ \Join $. 
\end{definition}

\begin{proposition}\label{prop:rt_ntf}
Every finite rooted tree is $ (\N,\Join) $-free. 
\end{proposition}
\begin{proof}
Let $ R $ be a finite rooted tree and $ r $ its root. 
Assume that $ R $ contains a full subposet which is isomorphic to $ \N $ or $ \Join $. 
Then there exist three elements $ a,b,c \in R $ such that $ a < b > c $ and the two elements $ a,c $ are incomparable. 
Let $p$ be a unique path from $ b $ to $ r $. 
The relation $ a < b > c $ implies that $p$ contains $ a $ and $ c $, which forces $ a<c $ or $ a>c $, a contradiction. 
\end{proof}

The class of $(\N,\Join)$-free posets is obviously contained in the class of $\N$-free posets. The latter class has a recursive characterization.  The class of \textbf{series-parallel posets} is the smallest class of finite posets (up to isomorphism) which contains the one-point poset $[1]$ and is closed under the disjoint union and the ordinal sum. A finite poset is $\N$-free if and only if it is series-parallel \cite[Appendix, Theorem 22]{NK98}. 

We prove below that the class of $ (\N,\Join) $-free posets has a similar characterization. We define a class $ \mathcal{C} $ as the smallest class of finite posets (up to isomorphism) which satisfies the following conditions:  
\begin{enumerate}[\rm(1)]
\item $ [1] \in \mathcal{C} $;  
\item If $ P, Q \in \mathcal{C} $ then $ P \sqcup Q \in \mathcal{C} $; 
\item If $ P \in \mathcal{C} $ then $ [1] \oplus P \in \mathcal{C}$ and $P \oplus [1] \in \mathcal{C} $. 
\end{enumerate}
These conditions, allowing repeated use of the same poset, give a recursive definition of posets in the class $\mathcal{C}$: starting from $[1]$, we build $[1] \sqcup [1]$ which is just the poset of two elements with no relation, $[1] \oplus [1]$ which is $[2]$, and so on. Two more examples are shown in Figs. \ref{fig:N-tie2} and \ref{fig:N-tie3}.  In this way we obtain all posets in $\mathcal{C}$. 

\begin{figure}
\begin{minipage}{0.5\hsize}
\begin{center}
\begin{tikzpicture}
\draw (5,0.5) node[v](5){}; 
\draw (6,0.5) node[v](6){};
\draw (5.5,1.5) node[v](7){};
\draw (5)--(7); 
\draw (6)--(7);
\node at (7.2,1) {= \quad (}; 
\draw (8,1) node[v](8){}; 
\draw (8.7,1) node[v](9){}; 
\node at (9.3,1) {) $\oplus$}; 
\draw (10,1) node[v](10){};
\end{tikzpicture}
\caption{The poset $\wedge = ([1]\sqcup[1]) \oplus [1]$}\label{fig:N-tie2}
\end{center}
\end{minipage}
\vspace{0mm}
\begin{minipage}{0.5\hsize}
\begin{center}
\begin{tikzpicture}
\draw (1,0.5) node[v](1){}; 
\draw (2,0.5) node[v](2){};
\draw (1.5,1.5) node[v](3){};
\draw (1)--(3); 
\draw (2)--(3);
\draw (0.2,0.5) node[v](0){}; 
\draw (0.2,1.5) node[v](-1){};
\draw (-1)--(0);
\node at (2.8,1) {=}; 
\draw (3.5,0.5) node[v](9){}; 
\draw (3.5,1.5) node[v](10){};
\draw (9)--(10);
\node at (4.3,1) {$\sqcup$}; 
\draw (5,0.5) node[v](5){}; 
\draw (6,0.5) node[v](6){};
\draw (5.5,1.5) node[v](7){};
\draw (5)--(7); 
\draw (6)--(7);
\end{tikzpicture}
\caption{The poset $[2] \sqcup \wedge$}\label{fig:N-tie3}
\end{center}
\end{minipage}
\end{figure}

 The main theorem of this section is the following.

\begin{theorem}\label{N-tie} Let $P$ be a finite poset. Then
$$
P\in \mathcal{C} ~\Longleftrightarrow ~\text{$P$ is $(\N,\Join)$-free}. 
$$
\end{theorem}
%\begin{corollary} 
%If $P\in\mathcal{C}$ and $Q$ is a full subposet of $P$ then $Q \in \mathcal{C}$. 
%\end{corollary}
To prove the theorem, we need an intermediate lemma. We say that a finite poset is \textbf{connected} if its Hasse diagram is connected. 
\begin{lemma}\label{conn}
A finite connected $(\N,\Join)$-free poset has a unique minimal or maximal element. 
\end{lemma}
\begin{proof}  
Let $P$ denote our poset, and consider a minimal element $u$ and a maximal element $v$. Since $P$ is connected there exists a shortest sequence $u< p_1 > p_2 < \cdots > p_{2 k} <v$ connecting $u$ and $v$, where $p_1,\dots, p_{2k}\in P$.  
Then the $\N$-freeness of $P$ forces $k=0$. This shows that $u<v$ for every minimal element $u$ and every maximal element $v$. If $P$ has two maximal elements and two minimal elements, then they form a full subposet of the form $\Join$, a contradiction. Therefore, $P$ must have a unique minimal or maximal element. 
\end{proof}

\begin{proof}[Proof of Theorem \ref{N-tie}] 
$(\Rightarrow)$ The proof is based on induction on $|P|$. If $P\in\mathcal{C}$ is not connected then $P$ is of the form $P' \sqcup P''$ for some $P',P'' \in \mathcal{C}$. 
By the induction hypothesis, $P', P''$ are $(\N,\Join)$-free, and so is $P$. If $P$ is connected then, by the definition of $\mathcal{C}$, $P$ is of the form $[1]\oplus P'$ or $P' \oplus [1]$ for some $P'\in\mathcal{C}$. We may thus assume without loss of generality that $P$ is of the form $P' \oplus [1]$ for some $P'\in\mathcal{C}$. The induction hypothesis shows that $P'$ is $(\N,\Join)$-free. This implies $P$ is $(\N,\Join)$-free too; otherwise we would find three elements $a,b,c$ in $P'$ such that the full subposet $(\{1,a,b,c\},\leq_P)$ is isomorphic to $\N$ or $\Join$, but this is a contradiction since $\N$ and $\Join$ do not have a unique maximal element. 

$(\Leftarrow)$ The proof is again based on induction on $|P|$. 
If a finite $(\N,\Join)$-free poset $P$ is not connected then $P$ is of the form $P' \sqcup P''$ for some non-empty $(\N,\Join)$-free posets $P'$ and $P''$. Induction hypothesis shows that $P'$ and $P''$ are in $\mathcal{C}$, and so is $P$. If $P$ is connected then, by Lemma \ref{conn}, we may assume without loss of generality that $P$ is of the form $P' \oplus [1]$ for some poset $P'$. Since $P'$ is a full subposet of $P$, it is $(\N,\Join)$-free and then induction hypothesis shows that $P'$ is in $\mathcal{C}$, and so is $P$.  
\end{proof}

\section{The proof of Theorem \ref{thm: main}} \label{proof of main theorem}
We prove Theorem \ref{thm: main} in a generalized form.  Recall that every rooted tree is $ (\N,\Join) $-free by Proposition \ref{prop:rt_ntf}. 
\begin{theorem}\label{thm:unique_C}
Let $ P,Q $ be $ (\N,\Join) $-free posets. 
Then the following are equivalent. 
\begin{enumerate}[\rm(1)]
\item\label{unique1} $ \Gamma^{<}(P,\boldsymbol{x}) = \Gamma^{<}(Q,\boldsymbol{x}) $. 
\item\label{unique2} $ \Gamma^{\leq}(P,\boldsymbol{x}) = \Gamma^{\leq}(Q,\boldsymbol{x}) $. 
\item\label{unique3} $ P $ and $ Q $ are isomorphic. 
\end{enumerate}
\end{theorem}
\begin{proof}
\eqref{unique1} and \eqref{unique2} are equivalent by Proposition \ref{prop:st_wk}. The implication $\eqref{unique3} \Rightarrow \eqref{unique1} $ is trivial. 
For the implication $\eqref{unique1} \Rightarrow \eqref{unique3}$ we use induction on $ |P| $. 
When $ |P| = 1 $, the degrees of the functions $ \Gamma^{<}(P,\boldsymbol{x}) $ and $ \Gamma^{<}(Q,\boldsymbol{x}) $ are $ 1 $. 
Hence $ |Q|=1 $. 
Thus $ P $ and $ Q $ are isomorphic. 

Assume that $ |P| \geq 2 $. 
Decompose $ P $ and $ Q $ into $ P = \sqcup_{i=1}^{n}P_{i} $ and $ Q = \sqcup_{i=1}^{m}Q_{i} $, where $ P_{i}, Q_{i} $ are non-empty connected subposets, which are $ (\N, \Join) $-free since $ P_{i} $ and $ Q_{i} $ are full subposets of $ P $ and $ Q $, respectively. 
By Lemma \ref{conn}, $P_{i}$ and $Q_{i}$ have unique minimal or maximal elements, and hence Lemma \ref{lem:min-max_irred} shows that $\Gamma^{<}(P_{i},\boldsymbol{x})$ and $ \Gamma^{<}(Q_{i},\boldsymbol{x}) $ are irreducible.
We obtain from Proposition \ref{lem:prod}\eqref{lem:prod1} the identity 
\begin{align*}
\prod_{i=1}^{n}\Gamma^{<}(P_{i},\boldsymbol{x})
=\prod_{i=1}^{m}\Gamma^{<}(Q_{i},\boldsymbol{x}). 
\end{align*}
By Corollary \ref{cor:UFD} and Proposition \ref{prop:lead_coeff}, we have $ n=m $ and $ \Gamma^{<}(P_{i},\boldsymbol{x}) = \Gamma^{<}(Q_{i},\boldsymbol{x}) $ for all $ i $ after a suitable renumbering. 
When $ n \geq 2 $, we have that $ |P_{i}|, |Q_{i}| < |P| $ and hence $P_i$ and $Q_i$ are isomorphic for every $ i $ by induction hypothesis. 
Hence $ P $ and $ Q $ are also isomorphic. 
Suppose that $ n=1 $, i.e., $ P $ and $ Q $ are connected. 
By Lemma \ref{conn}, $ P $ has a unique minimal or maximal element. 
If $ P $ has a unique minimal element, then $ Q $ also has a unique minimal element by Proposition \ref{prop:number_min_max}. 
Then we may express $ P=[1]\oplus P^{\prime} $ and $ Q=[1]\oplus Q^{\prime} $ for some posets $ P^{\prime},Q^{\prime} $. 
Since $ P^{\prime}, Q^{\prime} $ are full subposets of $ P, Q $, respectively, they are also $ (\N,\Join) $-free. 
Proposition \ref{lem:prod}\eqref{lem:prod2} shows that $M_{(1)}\ast\Gamma^<(P',\boldsymbol{x})=M_{(1)}\ast\Gamma^<(Q',\boldsymbol{x})$. It is then easy to see that the left factor $M_{(1)}$ may be cancelled out, so that $\Gamma^<(P',\boldsymbol{x})=\Gamma^<(Q',\boldsymbol{x})$. By the induction hypothesis,  $ P^{\prime} $ and $ Q^{\prime} $ are isomorphic.  
Therefore $ P $ and $ Q $ are also isomorphic. 
The case in which $ P $ has a unique maximal element is similar. 
\end{proof}

\section{Open problems}\label{open problem}
McNamara and Ward \cite[Figure 8]{mcnamaca2014equality} raised two finite posets which have the same strict order quasisymmetric function. 

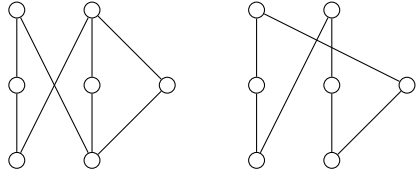
\begin{figure}[h] 
\centering
\begin{tikzpicture}
\draw (0,0) node[v](3){};
\draw (0,1) node[v](2){};
\draw (0,2) node[v](1){};
\draw (1,0) node[v](6){};
\draw (1,1) node[v](5){};
\draw (1,2) node[v](4){};
\draw (2,1) node[v](7){};
\draw (1)--(2)--(3)--(4)--(5)--(6)--(7);
\draw (1)--(6);
\draw (4)--(7);
\end{tikzpicture}
\qquad
\begin{tikzpicture}
\draw (0,0) node[v](3){};
\draw (0,1) node[v](2){};
\draw (0,2) node[v](1){};
\draw (1,0) node[v](6){};
\draw (1,1) node[v](5){};
\draw (1,2) node[v](4){};
\draw (2,1) node[v](7){};
\draw (1)--(2)--(3)--(4)--(5)--(6)--(7);
\draw (1)--(7);
\end{tikzpicture}
\caption{Two posets which have the same strict order quasisymmetric function.}\label{Poset}
\end{figure}
Using Proposition \ref{prop:expM}, one may compute the common strict order quasisymmetric function of posets in Fig.\ \ref{Poset}: 
\begin{align*}
& M_{232} + 2M_{2311} + 3M_{2221} + 3M_{2212} 
+ 9M_{22111} + M_{2131} + 3M_{2122} + 8M_{21211} \\
& + 7M_{21121} + 6M_{21112} + 20M_{211111} + M_{1321} + M_{1312} + 3M_{13111} + M_{1231}+ 3M_{1222} \\
& + 8M_{12211}+ 8M_{12121}+ 7M_{12112}+ 23M_{121111}+ 2M_{1132} + 4M_{11311} + 8M_{11221}+ 8M_{11212} \\
& + 24M_{112111} + 3M_{11131} + 9M_{11122} + 24M_{111211} + 23M_{111121} + 20M_{111112} + 66M_{1111111},  
\end{align*}
where the notation of compositions is simplified. 
%In particular, the strict order quasisymmetric functions do not distinguish all finite posets. 
We have shown that the strict order quasisymmetric functions distinguish rooted trees (Theorem \ref{thm: main}), and more generally, $(\N,\Join)$-free posets (Theorem \ref{thm:unique_C}). We propose two other natural classes, each of which does not contain the other.     
\begin{problem} \label{OT}
An \textbf{oriented tree} is a tree whose edges have orientations. An oriented tree is an acyclic digraph and hence has the natural poset structure, namely the relation $u<v$ holds if and only if there is a path from $u$ to $v$ along the orientation. Does the strict order quasisymmetric function distinguish oriented trees? How about series-parallel posets (namely, $\N$-free posets)? 
\end{problem}
The former and the latter problems generalize Theorems \ref{thm: main} and \ref{thm:unique_C} respectively. Note that the two posets in Fig.\ \ref{Poset} are neither oriented trees nor series-parallel posets, so they do not give a counterexample to the above problems.  

Another direction is to extend our result to labeled posets. 
\begin{problem}
Extend the class of $(\N,\Join)$-free posets to labeled posets and prove that $\Gamma(P,\omega, \boldsymbol{x})$ distinguishes the labeled posets in that class.  
\end{problem}

McNamara and Ward proposed a problem about irreducibility of $ (P,\omega) $-partition generating functions. 
\begin{problem}[McNamara-Ward {\cite[Question 7.2]{mcnamaca2014equality}}]
Let $ (P,\omega) $ be a connected labeled poset. 
Is $ \Gamma(P,\omega,\boldsymbol{x}) $ irreducible in $ \QSym $? 
\end{problem}
Irreducibility was very important in distinguishing finite rooted trees.  
We proved the irreducibility in the special case (Lemma \ref{lem:min-max_irred}) in which the poset has a unique minimal or maximal element and the labeling is strict. 
%\begin{problem}
%Is $ \Gamma^{<}(P,\boldsymbol{x}) $ irreducible in $ \QSym $ for a connected finite poset $ P $? 
%\end{problem} 

\begin{problem}
The sequence of the number of $ (\N,\Join) $-free posets with $n$ elements is $ 1, 2, 5, 14, 40$, $121, 373, 1184$ up to $n=8$. 
Does this sequence have any other interesting combinatorial interpretations? 
\end{problem}

%\begin{problem} Does Theorem \ref{thm:unique_C} have some consequence on the original Conjecture \ref{StaConj}? \end{problem}

Finally, we expect that our method based on algebraic properties (of the ring of quasisymmetric functions, in our case) has applications to other invariants for a variety of mathematical objects.

\section*{Acknowledgments} 
The authors thank Masahiko Yoshinaga for his interests, encouragements and comments, Peter McNamara for an interesting and fruitful suggestion that culminated in Theorem \ref{N-tie}, and referees for helpful suggestions. 

\bibliographystyle{amsplain1}
\bibliography{bibfile}

\end{document}